\documentclass[fleqn, a4paper, oneside, singlespacing]{imsart}

\usepackage{amsthm,amsmath,amssymb}

\usepackage[authoryear]{natbib}

\bibpunct{[}{]}{;}{a}{,}{,}
\RequirePackage[OT1]{fontenc}

\startlocaldefs

\newtheorem{theorem}{Theorem}
\newtheorem{lemma}{Lemma}
\newtheorem{proposition}{Proposition}
\newtheorem{corollary}[proposition]{Corollary}

\theoremstyle{remark}
\newtheorem{remark}{Remark}

\theoremstyle{definition}
\newtheorem{definition}{Definition}

\def\R{\mathbb{R}}
\def\N{\mathbb{N}}

\endlocaldefs

\begin{document}

\begin{frontmatter}

\title{Data-driven goodness-of-fit tests}
\runtitle{Data-driven tests}

\thankstext{t2}{Financial support of the Deutsche
Forschungsgemeinschaft GK 1023 "Identifikation in Mathematischen
Modellen" is gratefully acknowledged.}

\begin{aug}
\author{\fnms{Mikhail}
\snm{Langovoy}\thanksref{t2}\ead[label=e1]{mikhail.langovoy@epfl.ch}}

\affiliation{
         Machine Learning and Optimization Laboratory, EPFL, \\
         Lausanne, Switzerland.}

\address{Machine Learning and Optimization Laboratory, \\
EPFL IC, INJ 339, Station 14, \\
CH-1015  Lausanne, Switzerland\\
\printead{e1}\\}

\runauthor{M. Langovoy}
\end{aug}

\begin{abstract}
We propose and study a general method for construction of consistent statistical tests on the basis of possibly indirect, corrupted, or partially available observations. The class of tests devised in the paper contains Neyman's smooth tests, data-driven score tests, and some types of multi-sample tests as basic examples. Our tests are data-driven and are additionally incorporated with model selection rules. The method allows to use a wide class of model selection rules that are based on the penalization idea. In particular, many of the optimal penalties, derived in statistical
literature, can be used in our tests. We establish the behavior of model selection rules and data-driven tests under both the null hypothesis and the alternative hypothesis, derive an explicit detectability rule for alternative hypotheses, and prove a master consistency theorem for the tests from the class. The paper shows that the tests are applicable to a wide range of problems, including hypothesis testing in statistical inverse problems, multi-sample problems, and nonparametric hypothesis testing.\\
\end{abstract}



\end{frontmatter}

\section{Introduction}\label{GeneralTheorySection1}


Constructing good tests for statistical hypotheses is an essential
problem of statistics. There are two main approaches to constructing
test statistics. In the first approach, roughly speaking, some
measure of distance between the theoretical and the corresponding
empirical distributions is proposed as the test statistic. Classical
examples of this approach are the Cramer-von Mises and the
Kolmogorov-Smirnov statistics (\cite{Kolmogorov_Test_1933}). More generally, $L^p-$distance based
tests, as well as graphical tests based on confidence bands, usually
belong to this type, and so do kernel tests (see \cite{Dette_2002_Kernel_Test}) and RKHS-based kernel tests that recently became especially popular in machine learning (see, e.g., \cite{Sejdinovic_2013_Equivalence}, \cite{Zaremba_2013_B-test}, \cite{Chwialkowski_2016_Kernel_Test}). Although these tests are capable of giving very good results, each of these tests is asymptotically
optimal only in a finite number of directions of alternatives to a
null hypothesis (see \cite{MR1335235}).

Nowadays, there is an increasing interest to the second approach of
constructing test statistics. The idea of this approach is to
construct tests in such a way that the tests would be asymptotically
optimal in some sense, or most powerful, at least in a reach enough
set of directions. Test statistics constructed following this
approach are often called \emph{score test} statistics. The pioneer of this
approach was \cite{Neyman1937}. See also, for example,
\cite{0018.32003}, \cite{MR0084918}, \cite{MR0112201},
\cite{MR1226715}, \cite{MR1421157} for subsequent developments and
improvements, and \cite{1029.62042}, \cite{MR2281882} and \cite{liliang} for recent results in the field.
This approach is also closely related to the theory of efficient
(adaptive) estimation - \cite{MR1245941}, \cite{MR620321}.
Additionally, it was shown, at least in some basic situations, that
data-driven score tests are asymptotically optimal in the sense of
intermediate efficiency in an infinite number of directions of
alternatives (see \cite{MR1421157}) and show good overall
performance in practice (see \cite{MR1370299}).

Classical score tests have been extensively developed
in recent literature: see, for example, the generalized
likelihood ratio statistics for nonparametric models in
\cite{1029.62042}, tailor-made tests in \cite{MR2281882} and the
semiparametric generalized likelihood ratio statistics in
\cite{liliang}. See \cite{MR2281882} for a recent general overview of this and other
existing theories of statistical testing, and a discussion of some
advantages and disadvantages of different classes of testing
methods.

An important line of development in the area of optimal testing
concerns with minimax testing, see \cite{Ingster}, and adaptive
minimax testing, see \cite{Spokoiny}. Those tests are optimal in a
certain minimax sense against wide classes of nonparametric
alternatives. In \cite{1029.62042} it was shown that, for many
types of statistical problems, generalized likelihood ratio
statistics achieve optimal minimax and adaptive minimax rates.
Within the method of the present paper, one can construct and use different variations of data-driven
generalized likelihood ratio statistics, too. It is expected that these new tests would share the same optimality properties as the optimal tests from \cite{1029.62042}.

Historically, most of the systematic developments in statistical hypothesis testing were concerned with tests based on directly available observations, rather than with tests for those cases where the variables of interest are only available indirectly, and where the test has to be based on noisy, corrupted, or auxiliary third-party  variables. Examples of this type of problems arise in many modern applications, such as testing hypothesis about noisy signals in statistical inverse problems \cite{langovoy1}, \cite{Holzmann_2007_Density_Testing}, \cite{Bissantz_2007}, or hypothesis testing for stochastic processes. However, most of the existing solutions were proposed either on a case-by-case basis, or for relatively specific types of problems. This paper aims to develop a methodological approach that is applicable for both direct and indirect types of hypothesis testing problems.

\subsection{Main contributions}

In this paper we propose and study a general method for construction of consistent statistical tests on the basis of possibly indirect, corrupted, or partially available observations.

We introduce a class of \emph{data-driven} statistical \emph{NT-tests}. These tests are concerned with testing hypotheses about variables of interest $X_1, X_2, \ldots, X_m$ on the basis of auxiliary third-party variables $Y_1, Y_2, \ldots, Y_n$, and their corresponding test statistics are, in general, random quadratic forms of the type

\begin{equation*}
T_S =\biggr\{ \frac{1}{\sqrt{n}}\, \sum_{j=1}^{n} l (Y_j) \biggr\}
\;L\; {\biggr\{ \frac{1}{\sqrt{n}}\, \sum_{j=1}^{n} l (Y_j)
\biggr\} }^{T} \,,
\end{equation*}

\noindent where $l=(l_1, \ldots, l_S)$ is an appropriate vector-function with a random dimension chosen via a penalization-based selection rule $S$, and $L$ serves as a normalization matrix. The choice of $l$, $S$ and $L$ is problem-dependent and will be treated in this paper.

Neyman's smooth tests, data-driven score tests, some types of multi-sample tests, and efficient deconvolution density tests serve as examples for this class of tests.

The main goal of this paper is to derive a unified approach for construction of consistent NT-tests on the basis of possibly indirect or noisy observations, and in proposing a master theorem establishing consistency of NT-tests. We establish the behavior of model selection rules and data-driven tests under both the null and the alternative hypothesis, derive an explicit detectability rule for alternative hypotheses, and prove a master consistency theorem for data-driven NT-tests.

The paper shows that the tests are applicable to a wide range of problems, including hypothesis testing in statistical inverse problems, multi-sample problems, dependent data problems, and nonparametric hypothesis testing. Semi- and nonparametric generalized likelihood ratio statistics from \cite{1029.62042} and \cite{liliang}, score processes from \cite{MR2281882}, and empirical likelihood from \cite{MR946049}, could be used to build consistent data-driven NT-tests.

NT-tests are data-driven and are incorporated with model selection rules. Our method allows to use a wide class of model selection rules that are based on the penalization idea. In particular, many of the optimal penalties, derived in statistical literature, can be used in our tests. The method of this paper allows for big freedom in choice of penalties and dimension growth rates in data-driven model selection, and for flexibility in choosing models with convenient and realistic regularity assumptions. This allows to widen applicability of already existing data-driven tests.

It becomes possible to construct and properly tune
data-driven tests for new statistical problems by combining recent developments in model selection, efficient testing,
generalized likelihood, and other areas. We describe some of
the possibilities in this paper. A statistician can take into account
specific features of his particular problem and choose a test with desired properties among a family of
consistent data-driven tests.

In many modern problems, e.g. in semiparametric hypothesis testing, the model is so complicated that the calculations necessary for a proof of consistency  are hardly possible to carry out. Consistency theorems of this paper allow to bypass this step, and deliver consistency results in problems where direct computations are too cumbersome. In certain cases, our consistency theorems lead to generalization or improvement of already existing
results (see, for example, Proposition \ref{GeneralTheoryTh9} below).

For any data-driven NT-test we give an explicit rule to
determine, for every particular alternative, whether the test will
be consistent against this alternative. This rule allows us to
describe, in a closed form, the set of detectable alternatives for every
NT-test.

\subsection{Paper outline}

In Section \ref{GeneralTheorySection2}, we describe the framework
and introduce a class of NT-statistics, for the case of deterministic model dimension. This is one of the main building blocks for our tests.
In Section \ref{GeneralTheorySectionSelectionRule}, we define penalization-based model selection
rules that will be used in our tests, and introduce a concept of data-driven tests. In Section \ref{GeneralTheorySection3},
we establish behavior of model selection rules and data-driven NT-tests for the case when the alternative hypothesis is true. In this Section, we also derive an explicit consistency condition that allows to check whether any particular alternative could be detected by the data-driven NT-test.
In Section \ref{GeneralTheorySection4} we describe what
happens with a data-driven NT-test under the null hypothesis. Finally, in Section
\ref{GeneralTheorySection5} a master consistency theorem for data-driven NT-tests
is given. Additionally, each section contains examples illustrating the applicability of new concepts by connecting them to a variety of well-established special cases.

\section{NT-statistics}\label{GeneralTheorySection2}

\subsection{Basic definition: NT-statistics with deterministic dimensions}

Let $X_1, X_2, \ldots$ be a sequence of random variables with
values in an arbitrary measurable space $\mathbb{X}.$ Suppose that
for every $m$ the random variables $X_1,  \ldots , X_m$ have the
joint distribution $P_m$ from the family of distributions
$\mathbb{P}_m.$ Suppose there is a given functional $\mathcal{F}$
acting from the direct product of the families
$\otimes_{m=1}^{\infty}\, \mathbb{P}_m = (\mathbb{P}_1,
\mathbb{P}_2, \ldots)$ to a known set $\Theta ,$ and that
$\mathcal{F} (P_1, P_2, \ldots) = \theta .$ We consider the following generic \\

\noindent \textbf{Problem}: test the hypothesis $H_0: \quad \theta \in \Theta_0 \subset \Theta \quad$ against the alternative $H_A: \quad \theta \in \Theta_1 = \Theta \setminus \Theta_0$, on the basis of \emph{indirect observations} $Y_1,$ $\ldots ,$ $Y_n$
having their values in an arbitrary measurable space $\mathbb{Y}$
(i.e. \emph{not necessarily on the basis of variables of interest} $X_1,$ $\ldots ,$ $X_m$). \\

Here $\Theta$ can be any set, for example, a functional
space; correspondingly, parameter $\theta$ can be infinite
dimensional. It is not assumed that $Y_1,$ $\ldots,$ $Y_n$ are
independent or identically distributed. The measurable space
$\mathbb{Y}$ can be infinite dimensional. This allows
to apply the results of this paper in statistics for stochastic
processes. Additional assumptions on $Y_i'$s will be imposed below,
when it would be necessary.

An important feature of our approach is that we are able to consider
the case when the null hypothesis $H_0$ is not about observable $Y_i'$s, but about some
other random variables $X_1,$ $\ldots,$ $X_m .$ This makes it
possible to use our method in the case of statistical inverse
problems (see \cite{langovoy1}, and Examples 2 and 3 below). Under conditions of Theorem \ref{GeneralTheoryTh5} of the present paper, it would be still possible to extract from $Y_i'$s enough information in order to build a consistent test about $X_i'$s.

Now we introduce one of the main concepts of this paper.

\begin{definition}\label{GeneralTheoryDef3}\textbf{(NT-statistics with deterministic dimensions).}
Suppose we have $n$ random observations $Y_1,$ $\ldots ,$ $Y_n$
with values in a measurable space $\mathbb{Y}.$ Let $k$ be a fixed
number and $l=(l_1, \ldots, l_k)$ be a vector-function, where
$l_i:\, \mathbb{Y} \rightarrow \mathbb{R}$ for $i=1, \ldots , k$
are some known Lebesgue measurable functions. Set

\begin{equation}\label{GeneralTheory3}
L= {\{ E_0 {[l (Y)]}^{T} l (Y) \}}^{-1}\,,
\end{equation}

\noindent where the mathematical expectation $E_0$ is taken with
respect to $P_0,$ and $P_0$ is the (fixed and known in advance)
distribution function of some auxilliary random variable $Y,$ where
$Y$ is assuming its values in the space $\mathbb{Y}.$ Assume that
$E_0\,l(Y)=0$ and $L$ is well defined in the sense that all its
elements are finite. Put

\begin{equation}\label{GeneralTheory4}
T_k =\biggr\{ \frac{1}{\sqrt{n}}\, \sum_{j=1}^{n} l (Y_j) \biggr\}
\;L\; {\biggr\{ \frac{1}{\sqrt{n}}\, \sum_{j=1}^{n} l (Y_j)
\biggr\} }^{T} \,.
\end{equation}

\noindent We call $T_k$ a \emph{statistic of Neyman's type} (or
NT-statistic).
\end{definition}

Here $l_1,$ $\ldots,$ $l_k$ can be some score functions, as was the
case for the classical Neyman's test \cite{Neyman1937}, but it is possible to use any
other functions, depending on the problem under consideration. We
prove below that under additional assumptions it is possible to
construct consistent tests of such form without using scores in
(\ref{GeneralTheory4}). We will discuss different possible sets of
meaningful additional assumptions on $l_1,$ $\ldots,$ $l_k$ in Sections \ref{GeneralTheorySection3} -
\ref{GeneralTheorySection5}.

Scores are based on the notion of maximum likelihood. In our constructions it is possible to use, for example,
truncated, penalized or generalized likelihood to build a test. It is even
possible to use functions $l_1,$ $\ldots,$ $l_k$ such that
they are unrelated to any kind of likelihood.

If, for example, $Y_i'$s are equally distributed, then the natural
choice for $P_0$ is their distribution function under the null
hypothesis. Thus, $L$ will be the inverse to the covariance matrix
of the vector $l (Y).$ Such construction is often used in score
tests for simple hypothesis. But our definition allows to use a
reasonable substitution instead of the covariance matrix. This is helpful for testing in certain semi- or nonparametric models.\\

\subsection{Examples of NT-tests}

\noindent {\bf Example 1. Neyman's smooth tests.}\label{Example_1} The basic example of an NT-statistic is the
Neyman's smooth test statistic for simple hypotheses (see
\cite{Neyman1937}). Let $X_1,$ $\ldots ,$ $X_n$
be i.i.d. random variables. Consider the problem of testing the
simple null hypothesis $H_0$ that the $X_i'$s have the uniform
distribution on $[0,1].$ Let $\{\phi_j\}$ denote the family of
orthonormal Legendre polynomials on $[0,1].$ Then for every $k$ one
has the test statistic

\[ T_k =\, \sum_{j=1}^{k}{\biggr\{
\frac{1}{\sqrt{\,n}}\sum_{i=1}^{n} \phi_j (X_i) \biggr\}}^2 \,.\]

\noindent We see that Neyman's classical test statistic is
an NT-statistic, with $Y_i = X_i$ for all $i$ in this very simplest example. \\

Hypothesis testing in statistical inverse problems is an important but still a rather unexplored area in modern statistics. Applications to statistical inverse problems served as the main motivation for studying NT-tests. \\

\noindent {\bf Example 2. Statistical inverse problems.}\label{Example_2} The most
well-known example here is the deconvolution problem. This problem
appears when one has noisy signals or measurements: in physics,
seismology, optics and imaging, engineering. It is a building block
for many complicated statistical inverse problems. The book \cite{Carroll_etal} provides many examples related to deconvolution problems. In \cite{langovoy1}, data-driven score tests for the problem were constructed, thus solving the deconvolution density testing problem. The solution of \cite{langovoy1} is a special case of the results of the present paper.

The basic deconvolution density testing problem is formulated as follows. Suppose that instead of
$X_i$ one observes $Y_i,$ where
$$Y_i = X_i + \varepsilon_i ,$$ and $\varepsilon_i'$s are i.i.d.
with a known density $h$ with respect to the Lebesgue measure
$\lambda;$ also $X_i$ and $\varepsilon_i$ are independent for each
$i$ and $E\,\varepsilon_i =0,\,
 0 < E\,{\varepsilon}^2 < \infty .$ Assume that $X$ has a
density with respect to $\lambda.$ Our null hypothesis $H_0$ is
the simple hypothesis that $X$ has a known density $f_0$ with
respect to $\lambda.$ Let us choose for every $k \leq d(n)$ an
auxiliary parametric family $\{f_{\theta}\},\,$ $\theta \in \Theta
\subseteq \mathbb{R}^k$ such that $f_0$ from this family coincides
with $f_0$ from the null hypothesis $H_0.$ The true $F$ possibly
has no relation to the chosen $\{f_{\theta}\}.$ Set

\begin{equation}\label{score_deconvolution}
l (y) = \frac{ \frac{ \,\partial
}{\,\partial\theta}\,\Bigr(\int_{\R}\,f_{\theta}(s)\,h(\,y-s)\,ds\Bigr)\Bigr|_{\theta=0}}
{\int_{\R}\,f_0(s)\,h(\,y-s)\,ds}\,
\end{equation}

\noindent and define the corresponding test statistic $U_k$ by the
formula (\ref{GeneralTheory4}). Under appropriate regularity
assumptions, $U_k$ is an NT-statistic (see \cite{dissertation}). As was shown in \cite{langovoy1}, the data-driven test $U_S$, based on $U_k$, is asymptotically consistent against a wide class of nonparametric alternatives.

Subsequently, \cite{butucea2009adaptive} proposed an adaptive goodness-of-fit test for this model with indirect observations, which appears to be the first nonparametric adaptive test in statistical inverse problems. The test from \cite{butucea2009adaptive} is restricted to polynomially smooth error density and to signal densities coming from even narrower smoothness classes. The data-driven test $U_S$ from \cite{dissertation} was intended to cover wider scope of applications and to have a flexible penalization procedure, so optimality properties were not attainable over the intended consistency range. Adaptive modifications of $U_S$ can be constructed via using special penalties, tailored to deliver minimax optimality for convolution densities from properly restricted classes (analogously to the method of \cite{MR1833962}). \\

The following example serves as an illustration of applicability of NT-tests to multisample problems. \\

\noindent {\bf Example 3. Rank Tests for Independence.}\label{Example_3} Let $X_1 = (V_1,
W_1), \ldots , X_n = (V_n, W_n)$ be i.i.d. random variables with the
distribution function $D$ and the marginal distribution functions
$F$ and $G$ for $V_1$ and $W_1 .$ Assume that $F$ and $G$ are
continuous, but unknown. It is the aim to test the null hypothesis
of independence

\begin{equation}\label{GeneralTheory25}
H_0: \quad D (v, w) = F(v)G(w), \quad v,\, w \in \mathbb{R},
\end{equation}

\noindent against a wide class of alternatives. The following
construction was proposed in \cite{MR1689233}.

\noindent Let $\phi_j$ denote the $j-$th orthonormal Legendre
polynomial (i.e., $\phi_1 (x) = \sqrt{3} (2x-1),$ $\phi_2 (x)= \sqrt{5}
(6 x^2 - 6x +1),$ etc.). The score test statistic from
\cite{MR1689233} is

\begin{equation}\label{GeneralTheory26}
T_k = \, \sum_{j=1}^{k}{\biggr\{ \frac{1}{\sqrt{n}}\,
\sum_{i=1}^{n} \phi_j \biggr(\frac{R_i - 1/2}{n}\biggr) \phi_j
\biggr(\frac{S_i - 1/2}{n}\biggr) \biggr\} }^{2} \,,
\end{equation}

\noindent where $R_i$ stands for the rank of $V_i$ among $V_1,
\ldots , V_n$ and $S_i$ for the rank of $W_i$ among $W_1, \ldots ,
W_n .$ Thus defined $T_k$ satisfies Definition
\ref{GeneralTheoryDef3} of NT-statistics: put

$$Y_i = (Y_i^{(1)}, Y_i^{(2)}) := \biggr(\frac{R_i - 1/2}{n} ,
\frac{S_i - 1/2}{n}\biggr) $$ and $l_j (Y_i) := \phi_j (Y_i^{(1)})\,
\phi_j (Y_i^{(2)}).$ Under the null hypothesis $L_k = E_{k \times k},$
and $E_0 l(Z) = 0.$ Thus, $T_k$ is an NT-statistic. New $Y_i$
depends on the original $X_i'$s ($= (V_i, W_i)'$s) in a nontrivial way, but
still contains some information about the pair of interest. $\Box$

Many other examples, including data-driven tests based on partial likelihood
(see \cite{MR0400509}), can be found in \cite{dissertation} and \cite{langovoy_report_2009-007}.

\section{Data-driven tests and model selection}\label{GeneralTheorySectionSelectionRule}

\subsection{Selection rules and data-driven tests: definitions}

It was shown that for achieving optimality of efficient score tests it
is important to select the right number of components in the test
statistic (see \cite{MR1226715}, \cite{MR1248222}, \cite{MR1395735}, \cite{MR1964450}). Therefore, we provide a
corresponding refinement for NT-statistics as well. Using the idea of a
penalized likelihood, we propose a general mathematical framework
for constructing rules to find reasonable model dimensions. We make
our tests \emph{data-driven}, i.e., the tests choose a
reasonable number of components in the test statistics automatically and accordingly to the data. Our construction offers substantial freedom in the choice
of penalties and possible basis models. This way, it is easy to take into account specific features of particular problems
and to choose suitable statistical models and reasonable penalties in order to
build tests with desired properties for each particular problem.

We will not restrict a possible number of components in test
statistics by some fixed number, but instead we allow the number of
components to grow unlimitedly as the number of observations grows.
This is important because the more observations $Y_1,$ $\ldots ,$
$Y_n$ we have, the more information is available about the problem.
This makes it possible to give a more detailed description of the
phenomena under investigation. In our case this means that the
complexity of the model and the possible number of components in the
corresponding test statistic grow with $n$ at a controlled rate.

Denote by $M_k$ a statistical model designed for a specific
statistical problem satisfying assumptions of Section
\ref{GeneralTheorySection2}. Assume that the true parameter value
$\theta$ belongs to the parameter set of $M_k$, call it $\Theta_k .$
We say that the family of models $M_k$ for $k=1,2, \ldots$ is nested
if for their parameter sets it holds that $\Theta_1 \subseteq
\Theta_2 \subseteq \ldots .$ Here $\Theta_k'$s can be
infinite-dimensional. We also do not require that all $\Theta_k'$s are
different (analogously to \cite{MR1848946}, p. 221).

\noindent Let $T_k$ be an \emph{arbitrary} statistic for testing
validity of the model $M_k$ on the basis of observations $Y_1,
\ldots , Y_n.$ The following definition applies for the sequence of
statistics $\{T_k\}.$

\begin{definition}\label{GeneralTheoryDef2}\textbf{(Penalty, selection rule, and data-driven tests)}
Consider a nested family of models $M_k$ for $k=1,$ $\ldots ,$
$d(n),$ where $d(n)$ is a control sequence, giving the largest
possible model dimension for the case of $n$ observations. Choose a
function $\pi (\cdot,\cdot): \N \times \N \rightarrow \R,$ where
$\N$ is the set of natural numbers. Assume that $\pi(1,n) <\pi(2,n)<
\ldots <\pi(d(n),n)$ for all $n$ and $\pi(j,n)-\pi(1,n) \rightarrow
\infty$ as $n \rightarrow \infty$ for every $j= 2, \ldots ,d(n).$
Call $\pi(j,n)$ a \emph{penalty attributed to the $j$-th model $M_j$
and the sample size n.}

Then a \emph{selection rule} $S$ for the
sequence of statistics $\{T_k\}$ is an integer-valued random
variable satisfying the condition

\begin{equation}\label{GeneralTheory2}
S=\min\bigr\{ k:\,1\leq k \leq d(n);\, T_k - \pi(k,n) \geq T_j
-\pi(j,n),\, j= 1, \ldots ,d(n)\bigr\}\,.
\end{equation}

\noindent We call $T_S$ a \emph{data-driven test statistic} for
testing validity of the initial model.
\end{definition}

The definition is expected to be applied, of course, only if the sequence $\{T_k\}$ is increasing unboundedly, in the sense
that $T_k (Y_1, \ldots , Y_n) \rightarrow \infty$ in probability as $k \rightarrow \infty .$

In statistical literature, one often tries to choose penalties
such that they possess some sort of minimax or Bayesian optimality.
Classical examples of the penalties constructed via this approach
are Schwarz's penalty $\pi(j,n)=j \, \log n$ (see \cite{MR0468014}),
and minimum description length penalties, see \cite{Rissanen}. For
more examples of optimal penalties and recent developments, see
\cite{AGP}, \cite{MR1848946} or \cite{BTW}. In this paper, we do not
aim for optimality of the penalization; our goal is to be able to
build consistent data-driven tests based on different choices of
penalties. The penalization technique that we use in this paper allows
for many possible choices of penalties. In our
framework it is possible to use most of the penalties from the
abovementioned papers. As an illustration, see Example 4 below.

\subsection{Model selection rules: examples}

\noindent {\bf Example 2, continued. Data-driven tests for inverse problems.}\label{Example_2}
One can incorporate $U_k$ with the following selection rule widely used in statistical literature:

\begin{equation}\label{GeneralTheory27}
S=\min\bigr\{ k:\,1\leq k \leq d(n);\, T_k - k \log n \geq T_j - j
\log n,\, j= 1,2, \ldots ,d(n)\bigr\}\,.
\end{equation}

\noindent This selection rule satisfies Definition
\ref{GeneralTheoryDef2}, so thus defined $U_S$ extends the test statistic
from \cite{langovoy1} and makes it a data-driven NT-statistic. $\Box$\\

\noindent {\bf Example 4. Gaussian model selection.} Birg\'{e} and
Massart in \cite{MR1848946} proposed a method of model selection in
a framework of Gaussian linear processes. This framework is quite
general and includes as special cases a Gaussian regression with
fixed design, Gaussian sequences and the model of Ibragimov and
Has'minskii. In this example we briefly describe the construction
(for more details see the original paper) and then discuss the
relation with our results.

Given a linear subspace $\mathbb{S}$ of some Hilbert space
$\mathbb{H},$ we call Gaussian linear process on $\mathbb{S},$ with
mean $s \in \mathbb{H}$ and variance $\varepsilon^2,$ any process
$Y$ indexed by $\mathbb{S}$ of the form

$$Y(t) = \langle s, t \rangle + \varepsilon Z(t),$$ for all $t \in
\mathbb{S},$ and where $Z$ denotes a linear isonormal process
indexed by $\mathbb{S}$ (i.e. $Z$ is a centered and linear
Gaussian process with covariance structure $E[Z(t)Z(u)]= \langle
t, u \rangle$). Birg\'{e} and Massart considered estimation of $s$
in this model.

\noindent Let $S$ be a finite dimensional subspace of $\mathbb{S}$
and set $\gamma (t) = \|t\|^2 - 2 Y(t).$ One defines the projection
estimator on $S$ to be the minimizer of $\gamma (t)$ with respect to
$t \in S.$ Given a finite or countable family $\{S_m\}_{m \in
\mathcal{M}}$ of finite dimensional linear subspaces of $S,$ the
corresponding family of projection estimators $\widehat{s}_m,$ built
for the same realization of the process $Y,$ and given a nonnegative
function $pen$ defined on $\mathcal{M},$ Birg\'{e} and Massart
estimated $s$ by a penalized projection estimator $\widetilde{s} =
\widehat{s}_{\widehat{m}},$ where $\widehat{m}$ is any minimizer
with respect to $m \in \mathcal{M}$ of the penalized criterion

$$crit (m) = - \| \widehat{s}_m \|^2 + pen (m) = \gamma (\widehat{s}_m) + pen (m) .$$

\noindent They proposed some specific penalties $pen$ such that the
penalized projection estimator has the optimal order risk with
respect to a wide class of loss functions. The method of model
selection of this paper has a relation with the one of
\cite{MR1848946}.

\section{Data-driven NT-tests: alternative hypothesis case}\label{GeneralTheorySection3}

In this section we study the behavior of NT-statistics and model selection rules under alternatives.

\subsection{Explicit consistency condition}

Let $P$ denote the alternative distribution of $Y_i'$s. Suppose that
$E_P\,l(Y)$ exists.

\noindent Now we formulate the following crucial explicit \emph{consistency
condition}: \\

\noindent \textbf{Consistency condition}

$\quad$

$\langle \mathbf{C} \rangle \quad\quad\mbox{there exists
integer}\,\, K = K(P)\geq 1 \,\,\mbox{such that}$

$\quad\quad\quad\quad E_P\,l_1(Y)=0, \, \ldots ,
E_P\,l_{K-1}(Y)=0,\,\,E_P\,l_K (Y) = C_P \neq 0\,,$

$\quad$

\noindent where $l_1,$ $\ldots ,$ $l_k$ are as in Definition
\ref{GeneralTheoryDef3}.\\

\noindent {\bf Example 1 (continued). Consistent smooth tests.}  In this basic problem, condition $\langle C \rangle$ is simply equivalent to
the following one: there exists $K = K_{\mathbb{P}}$ such that $E_{\mathbb{P}}\, \phi_K (X) \neq 0$. $\Box$
\\

\noindent {\bf Example 2 (continued). Consistent density testing in inverse problems.}  Let $F$ be a true distribution of $X$. Here $F$ is not supposed to be parametric and possibly does not have a density. Let $l$ be the score vector given by (\ref{score_deconvolution}) and $h$ be the known error density. As was shown in \cite{langovoy1} and \cite{dissertation}, if there exists an integer $K = K(F)\geq 1$ such that

\begin{equation}\label{score_deconvolution_consistency}
E_{F \ast h}\,l_1(Y)=0, \, \ldots ,
E_{F \ast h}\,l_{K-1}(Y)=0,\,\,E_{F \ast h}\,l_K (Y) \neq 0\,,
\end{equation}

\noindent then $U_S$ is asymptotically consistent against a wide class of nonparametric alternatives. $\Box$ \\

\noindent We see that conditions equivalent to $\langle C \rangle$ naturally appear in studies of consistency of data-driven score tests. This condition explicitly describes which directions of alternatives to the null hypothesis can be captured by an NT-test.

\subsection{Model selection under alternatives}

We impose now two additional assumptions on the abstract model of Section
\ref{GeneralTheorySection2}. First, we assume that $Y_1, Y_2,
\ldots$ are identically distributed. We do \emph{not} assume that
$Y_1, Y_2, \ldots$ are independent. It is possible that the sequence
of interest $X_1, X_2, \ldots$ consists of dependent and
nonidentically distributed random variables.

It is only important that the new (possibly obtained by a complicated transformation)
sequence $Y_1, Y_2, \ldots$ obeys the consistency conditions. Then
it is possible to build consistent tests of hypotheses about
$X_i'$s. The intuition behind this is that, even after a complicated
transformation, the transformed sequence still can contain some part
of the information about the sequence of interest. However, if the
transformed sequence $Y_1, Y_2, \ldots$ is not chosen reasonably,
then test can be meaningless: it can be formally consistent, but
against an empty or almost empty set of alternatives.

Another assumption we impose is that the first $K$ components of $l(Y_i)'$s
satisfy both the law of large numbers and the multivariate central
limit theorem, i.e. that for the vectors $l(Y_1) = (l_1 (Y_1), \ldots, l_K (Y_1)),$ $\ldots ,$
$l(Y_n) = (l_1 (Y_n), \ldots, l_K (Y_n))$ it holds that

\[
\frac{1}{n}\sum_{j=1}^{n} l(Y_j) \rightarrow E_P\,l(Y)
\quad\mbox{in}\,\,P-\mbox{probability as}\,\,n \rightarrow \infty
,
\]

\begin{equation}\label{GeneralTheory34}
n^{-1/2} \sum_{j=1}^{n} ( l(Y_j) - E_P\,l(Y) ) \rightarrow_{d}
\mathcal{N} 
\end{equation}

\noindent where $\mathcal{N}$ stands for a $K-$dimensional normal
distribution with zero mean vector.

\noindent These assumptions restrict the choice
of the function $l$ and leave us with a uniquely determined $P_0.$
Nonetheless, random variables of interest $X_1,$ $\ldots ,$ $X_n$
are still allowed to be nonidentically distributed, and their transformed counterparts $Y_1,$ $\ldots ,$ $Y_n$
are still allowed to be dependent.

\noindent Without loss of generality, below we always assume that

\begin{equation}\label{GeneralTheory5}
\lim_{n \rightarrow \infty}\,d(n)\,=\,\infty \,.
\end{equation}

\noindent The (relatively degenerate) case when $d(n)$ is nondecreasing and bounded from above by some constant $D$ can be handled analogously to the
method of this paper, see \cite{dissertation}.

To clarify the statement of the next theorem, we remind that in Definition \ref{GeneralTheoryDef3} $L$
is a $k \times k-$matrix. Below we will sometimes denote it
$L_k$ in order to stress the model dimension. Accordingly, ordered
eigenvalues of $L_k$ will be denoted $\lambda^{(k)}_1 \geq
\lambda^{(k)}_2 \geq \ldots \geq \lambda^{(k)}_k.$ We have the
sequence of matrices ${\{L_k\}}^{\infty}_{k=1}$ and each matrix has
its own eigenvalues. Whenever possible, we will use the
simplified notation from Definition \ref{GeneralTheoryDef3}.

\begin{proposition}\label{GeneralTheoryTh2}
Let $\langle C \rangle$ and (\ref{GeneralTheory5}) holds and

\begin{equation}\label{GeneralTheory6}
\lim_{n \rightarrow \infty} \sup_{k \leq d(n)} \frac{\,\pi(k,n)}{n
\lambda^{(k)}_k} = 0\,.
\end{equation}

\noindent Then

$$\lim_{n \rightarrow \infty}\,P(S \geq K)=1\,.$$
\end{proposition}

\begin{remark}\label{GeneralTheoryRemark2} Condition (\ref{GeneralTheory6})
means that not only $n$ tends to infinity, but that it is also
possible for $k$ to grow infinitely, but at the controlled rate. This condition
is needed to ensure that the penalty $\pi$ is not so large that it trivializes
the data-driven model choice. It is easy to check (\ref{GeneralTheory6}) for
any specific function $\pi$ since the eigenvalues of $L_k$ are known in advance,
and it is easy to satisfy (\ref{GeneralTheory6}) because we are free to choose
$\pi$ from the whole range of possible penalty functions.
\end{remark}

\subsection{Data-driven NT-tests under detectable alternatives}

Now suppose that the alternative distribution $P$ is such that
$\langle C \rangle$ is satisfied and that there exists a sequence
${\{r_n\}}_{n=1}^{\infty}$ such that $\lim_{n \rightarrow
\infty}\,r_n=\infty $ and

$$\langle \mathbf{A} \rangle \quad\quad\quad\quad P\,\biggr( \frac{1}{n}\,\biggr|
\sum_{i=1}^{n} \bigr[l_K (Y_i) - E_P\,l_K (Y_i)\bigr]\biggr| \geq
y  \biggr)\,=\,O \biggr(\frac{1}{\,r_n}\biggr)\,.$$

\noindent Note that in $\langle A \rangle$ we do not require
uniformity in $y,$ i.e. $r_n$ gives us the rate, but the exact
bound can depend on $y.$ In some sense condition $\langle A
\rangle$ is a way to make the weak law of large numbers for $l_K
(Y_i)'$s more precise. The following lemma shows that condition $\langle A \rangle$ is often easily satisfied.

\begin{lemma}\label{GeneralTheoryLemma2}
Let $l_K (Y_i)'$s be bounded i.i.d. random variables with finite
expectation and variance $\sigma^2.$ Then condition $\langle A
\rangle$ is satisfied with $r_n = \exp (n y^2 / 2 \sigma).$
\end{lemma}

\noindent Therefore, one can often expect exponential rates in
$\langle A \rangle ,$ but even a much slower rate is not a problem. The proof of this standard lemma is omitted.

The main result of this section is the following theorem that describes asymptotic behavior of NT-tests under alternative hypotheses.

\begin{theorem}\label{GeneralTheoryTh3}
Let $\langle A \rangle,$ $\langle C \rangle,$
(\ref{GeneralTheory5}) and (\ref{GeneralTheory6}) holds and

\begin{equation}\label{GeneralTheory11}
d(n)=o(r_n) \quad\mbox{as}\quad n \rightarrow\infty \,.
\end{equation}

\noindent Then $T_S \,\rightarrow_P \, \infty$ as $n
\rightarrow\infty \,.$
\end{theorem}

\begin{remark}\label{DistributionSNTStatistics} Sometimes it is possible to find asymptotic distributions of
NT-statistics, both under the null hypothesis and under
alternatives. In order to do this, one can approximate
the quadratic form $T_k$ by certain quadratic form in $N_1, \ldots , N_k$, where $N_i$'s
are Gaussian random variable with suitably chosen mean and covariance
structure. When this approximation is possible, the error of approximation is often of order
$n^{-1/2}$ and depends on the smallest eigenvalue of the covariance
of $l(Y_1).$ See \cite{MR1699003}, p. 1078 for more details.

Note that for proving consistency of data-driven NT-tests one doesn't need to know asymptotic distributions of underlying NT-statistics. Theorem \ref{GeneralTheoryTh3} is strong enough for our purposes.
\end{remark}

\section{Data-driven NT-tests under the null hypothesis}\label{GeneralTheorySection4}

Now we study the asymptotic behavior of data-driven NT-statistics
under the null hypothesis.

\subsection{Proper penalties and majorants for data-driven tests}

The following two technical definitions are useful for data-driven model selection in general.

\begin{definition}\label{GeneralTheoryDef4}\textbf{(Penalties of proper weight)}.
Let $\{ T_k \}$ be a sequence of NT-statistics and $S$ be a
selection rule for it. Suppose that $\lambda_1 \geq \lambda_2 \geq
\ldots$ are ordered eigenvalues of $L,$ where $L$ is defined by
(\ref{GeneralTheory3}). We say that the penalty $\pi (k,n)$ in $S$
is of \emph{proper weight}, if the following conditions holds:

\begin{enumerate}
\item there exists sequences of real numbers ${\{ s(k,n)
\}}_{k,n=1}^{\infty}\,,$ ${\{ t(k,n) \}}_{k,n=1}^{\infty}\,,$ such
that

\begin{enumerate}
\item
\[\lim_{n \rightarrow \infty} \sup_{k \leq u_n}
\frac{\,s(k,n)}{n \lambda^{(k)}_k} = 0 \,,\] where
${\{u_n\}}_{n=1}^{\infty}$ is some real sequence such that
$\lim_{n \rightarrow\infty} u_n=\infty .$

\item $\lim_{n \rightarrow\infty} t(k,n) =\infty $ for every $k
\geq 2$,

$\lim_{k \rightarrow\infty} t(k,n) =\infty $ for every fixed $n.$
\end{enumerate}

\item $s(k,n) \leq \pi (k,n) - \pi (1,n) \leq t(k,n)$ for all $k,$
$n$.

\item \[\lim_{n \rightarrow \infty} \sup_{k \leq m_n}
\frac{\,\pi(k,n)}{n \lambda^{(k)}_k} = 0 \,,\] where
${\{m_n\}}_{n=1}^{\infty}$ is some real sequence such that
$\lim_{n \rightarrow\infty} m_n=\infty .$

\end{enumerate}
\end{definition}

\noindent Definition \ref{GeneralTheoryDef4} is more general
than required for our present goals. As can be seen from Theorem
\ref{GeneralTheoryTh4} below, in our research on consistency of
data-driven tests we can assume $u_n = m_n = d(n)$ for all
$n$, where $d(n)$ is the model complexity growth rate from Definition
\ref{GeneralTheoryDef2}. We have chosen to state Definition \ref{GeneralTheoryDef4}
in its full generality because of its general usefulness in model
selection.

Condition 2 of Definition \ref{GeneralTheoryDef4} is needed to ensure
that the penalty $\pi$ actually penalizes models with large dimension
$k$. The choice of $s(k,n)$ and $t(k,n)$ will be explained below.

\noindent For notational convenience, for $l = (l_1, \ldots , l_k)$
from Definition \ref{GeneralTheoryDef3} we denote

\begin{equation}\label{GeneralTheory12}
\overline{l}_j := \frac{1}{n\,}\,\sum_{i=1}^{n} l_j (Y_i)\,,\quad\quad
\overline{l} := (\overline{l}_1, \overline{l}_2, \ldots ,
\overline{l}_k)\,
\end{equation}

\noindent and define a quadratic form

\begin{equation}\label{GeneralTheory15}
Q_k (\overline{l}) = (\overline{l}_1, \overline{l}_2, \ldots ,
\overline{l}_k)\, L\, {(\overline{l}_1, \overline{l}_2, \ldots ,
\overline{l}_k)}^{T} \,.
\end{equation}

\noindent In the new notation, $T_k = Q_k (\overline{l}),$ where $T_k$ is the statistic from Definition
\ref{GeneralTheoryDef3}. Below we also use the notation of
Definition \ref{GeneralTheoryDef4}.

\begin{definition}\label{GeneralTheoryDef5}\textbf{(Proper majorant)}.
Let $S$ be with a penalty of proper weight. Assume that there
exists a Lebesgue measurable function $\varphi (\cdot ,\cdot):\,
\R \times \mathbb{R} \rightarrow \mathbb{R},$ such that $\varphi$
is monotonically decreasing in the second argument and
monotonically nondecreasing in the first one, and assume that

\begin{enumerate}

\item for every $\varepsilon > 0$ there exists
$K=K_{\varepsilon}$ such that  for every $n > n(\varepsilon)$

$$\!\!\!\!\!\!\!\!\!\!\!\!\!\!\!\!\!\!\!\!\!\!\!\!\!\!\!\!\!\!\!\!\!\!\!\!\!\!\!\!\langle \mathbf{B1} \rangle \quad\quad\quad\quad\quad\quad\sum_{k=K_{\varepsilon}}^{u_n} \varphi (k; s(k,n)) \,<\,\varepsilon \,,$$

\noindent where ${\{u_n\}}_{n=1}^{\infty}$ is as in Definition
\ref{GeneralTheoryDef4}.

\item

$$\!\!\!\!\!\!\!\!\!\!\!\!\!\!\!\!\!\!\!\!\!\!\!\!\!\!\!\!\!\!\!\!\!\!\!\!\langle \mathbf{B2} \rangle \quad\quad\quad\quad\quad P_0 \,(n\,Q_k (\overline{l}) \geq y)\,\leq\, \varphi
(k; y)\,$$ for all $k \geq 1$ and $y \in [s(k,n); t(k,n)]\,,$
where $P_0$ is as in Definition \ref{GeneralTheoryDef3}.

\end{enumerate}

\noindent We call $\varphi$ a \emph{proper majorant} for (large
deviations of) the statistic $T_S.$ Equivalently, we say that (large
deviations of) the statistic $T_S$ are \emph{properly majorated} by
$\varphi.$
\end{definition}

\noindent Again, one is free to assume in this Definition that
$u_n = d(n)$ for all $n$. In Definitions \ref{GeneralTheoryDef4}
and \ref{GeneralTheoryDef5} we need some lower bound $s(k,n)$ to be
sure that the penalty $\pi$ is not "too light", i.e. that the penalty
somehow affects the choice of the model dimension and protects us
from choosing a "too complicated" model. In nontrivial cases, it
follows from $\langle B1 \rangle$ that $s(k,n) \rightarrow \infty$ as $k \rightarrow
\infty .$ The exact choice of functions $s$ and $t$ for each specific
problem is dictated by the variant of inequalities $\langle B1 \rangle$ and $\langle B2 \rangle$ that
holds in this particular statistical problem. \\

\noindent {\bf Example 1 (continued). Model selection for data-driven smooth tests.} Put $\pi (k,n) = k\,\log n$. Then Prohorov's inequality (see Appendix) suggests the following possible $\varphi$ in Definition \ref{GeneralTheoryDef5}:

\begin{equation}\label{Neyman_Majorant}
\varphi (k; y)\,=\, \frac{150210}{\Gamma (k/2)}\,
{\biggr(\frac{y^2}{2}\biggr)}^{\frac{k-1}{2}} \exp \biggr\{ - \frac{y^2}{4} \biggr\}\,,
\end{equation}

\noindent with, for example, $s(k,n):= 3k$ and $t(k, n):= 1/4 \sqrt{\,n/k\,}$. All conditions of Definitions \ref{GeneralTheoryDef4} and \ref{GeneralTheoryDef5} can be easily checked. It follows that one can afford any model dimension growth rate such that $d(n) = o ({\{n/\log n\}}^{1/3})$.  $\Box$

For many test statistics Prohorov's inequality can't be used
to estimate the large deviations. This is usually the case for
more complicated models where the matrix $L$ is not diagonal. This
is typical for statistical inverse problems as well as for problems with dependent data. If one deals with likelihood-based NT-statistics, then the easiest way is to use large deviations results that were established in studies on usual likelihood tests and their efficiencies. Otherwise, one has to find a direct way to estimate the test statistic.

\begin{remark}\label{Remark_Quadratic_Forms}
There is a general approach that makes it possible to check conditions of
Definition \ref{GeneralTheoryDef5}, and to study distributions of NT-statistics for moderate sample sizes, in many particular situations. This
method consists of two steps. On the first step, one approximates the
quadratic form $Q(\overline{l}(Y))$ by the simpler quadratic form $Q(N),$
where $N$ is the Gaussian random variable with the same mean and covariance
structure as $l(Y).$ This approximation is possible, for example,
under conditions given in \cite{MR1699003} or \cite{MR1387646}.
These authors gave the rate of convergence for such approximation.
The second step is to establish a large deviation result for
the quadratic form $Q(N);$ this form has a more predictable
distribution. Even for strongly dependent random variables, one can hope
to use results from \cite{MR1682588}.
\end{remark}

\subsection{Data-driven NT-statistics under null hypotheses}

The following is the key result describing the behaviour of a data-driven NT-statistic under the null hypothesis.

\begin{theorem}\label{GeneralTheoryTh4}
Let $\{ T_k \}$ be a sequence of NT-statistics and $S$ be a
selection rule for the sequence. Assume that the penalty in $S$ is of proper
weight and that large deviations of statistics $T_S$ are properly
majorated. Suppose that

\begin{equation}\label{GeneralTheory16}
d(n)\,\leq\, \min \{u_n \,, m_n\}\,.
\end{equation}

\noindent Then $S= O_{P_0}(1)$ and $T_S=O_{P_0}(1).$
\end{theorem}

\begin{remark}\label{GeneralTheoryRemark3}
Usually it happens that if the penalty does not satisfy all the
conditions of Definitions \ref{GeneralTheoryDef4} and
\ref{GeneralTheoryDef5}, then $T_S$ has the same distribution under
both alternative and null hypotheses and the test is inconsistent.
\end{remark}

As the first corollary, we have the following generalization of previously known null-hypothesis results for Neyman's smooth test. As was shown in \cite{dissertation} and \cite{langovoy_report_2009-007}, this corollary allows to build statistical tests for new problems that go far beyond the classic case of testing simple hypothesis of uniformity.

\begin{proposition}\label{GeneralTheoryTh9}
Let $\{T_k\}$ be a family of NT-statistics and $S$ a selection rule
for the family. Assume that $Y_1,$ $\ldots ,$ $Y_n$ are i.i.d. Let
$E\,l(Y_1)=0$ and assume that for every $k$ the vector $(l_1(Y_i),
\ldots ,l_k(Y_i))$ has the unit covariance matrix. Suppose that $\|
(l_1(Y_1), \ldots ,l_k(Y_1)) \|_k \leq M(k)$ a.e., where $\| \cdot
\|_k$ is the norm of the $k-$dimensional Euclidean space. Assume
$\pi (k,n) - \pi (1,n) \geq 2k$ for all $k \geq 2$ and

\begin{equation}\label{GeneralTheory39}
\lim_{n \rightarrow \infty}\, \frac{M(d(n))\, \pi (d(n),
n)}{\sqrt{n}}=\,0.
\end{equation}

\noindent Then $S=O_{P_0}(1)$ and $T_S=O_{P_0}(1).$
\end{proposition}

\noindent {\bf Example 1 (continued). Data-driven smooth tests under the null hypothesis.} As a simple corollary, we
derive the following slight improvement of Theorem 3.2
from \cite{MR1964450}.

\begin{corollary}\label{GeneralTheoryTh6}
Let $T_S$ be the Neyman's smooth data-driven test statistic for
the case of simple hypothesis of uniformity. Assume that $\pi
(k,n)- \pi (1,n) \geq 2k$ for all $k \geq 2$ and that for all $k
\leq d(n)$

$$\lim_{n \rightarrow \infty}\, \frac{d(n) \pi
(d(n),n)}{\sqrt{\,n}}\,=\,0.$$ Then $S=O_{P_0}(1)$ and
$T_S=O_{P_0}(1).$
\end{corollary}

\begin{proof}
It is enough to note that in this case $M(k)= \sqrt{(k-1)(k+3)}$
and apply Proposition \ref{GeneralTheoryTh9}.
\end{proof}

To prove consistency of a test based on some test statistic, usually
it is required to use inequalities for large deviations of this test
statistic. NT-statistics are no exception: in order to prove
consistency of an NT-test, one first has to establish large
deviations inequalities for NT-statistics under consideration. The model regularity assumptions and the model dimension growth rate
$d(n)$ will be strongly connected with this inequality.

In certain cases it can be desirable to have $d(n)$ growing at
a fast rate, even by the cost of restrictive regularity assumptions
that can arise from using a sharp large deviations inequality.
Sometimes, it is better to use a simple crude inequality that requires
only little additional assumptions, but allows only for a slow growth
of $d(n).$ One useful point of Definitions \ref{GeneralTheoryDef4} and
\ref{GeneralTheoryDef5} and Theorem \ref{GeneralTheoryTh5} below is
that one can be sure in advance that whatever large deviations
inequality he is able to prove, he will succeed in proving a
consistency theorem, provided that the chosen inequality satisfies
conditions $\langle B1 \rangle$ and $\langle B2 \rangle$. Moreover, once an inequality is chosen, the
rate $d(n)$ is obtained from Theorem \ref{GeneralTheoryTh5}.

\section{Consistency theorem}\label{GeneralTheorySection5}

Now we formulate the general consistency theorem for data-driven
NT-tests. We understand consistency of the test based on
$T_S$ in the sense that under the null hypothesis $T_S$ is bounded
in probability, while under fixed alternatives $T_S \rightarrow
\infty$ in probability.

\begin{theorem}\label{GeneralTheoryTh5}\textbf{(Consistency of data-driven NT-tests.)}
Let $\{T_k\}$ be a sequence of NT-statistics and $S$ be a
selection rule for it. Assume that the penalty in $S$ is of proper
weight. Assume that conditions (A), (\ref{GeneralTheory5}) and
(\ref{GeneralTheory6}) are satisfied and that $d(n)=o(r_n),$ $d(n)
\leq \min \{ u_n, m_n \}.$ Then the test based on $T_S$ is
consistent against any alternative distribution $P$
satisfying condition (C).
\end{theorem}


Without a general consistency theorem, one has to perform the whole
proof of consistency of a data-driven test anew for every particular problem. This becomes
especially difficult in many semi- and nonparametric problems. Using the general
consistency Theorem \ref{GeneralTheoryTh5}, some type of
consistency result can be obtained for any data-driven NT-statistics, and types of detectable alternatives can be characterized explicitly. \\

\noindent {\bf Example 4 (continued). Data-driven tests for Gaussian linear processes.}
In the model of Birg\'{e} and Massart $\gamma (t)$ is the least
squares criterion and $\widehat{s}_m$ is the least squares estimator
of $s,$ which is in this case the maximum likelihood estimator.
Therefore $\| \widehat{s}_m \|^2$ is the Neyman score for testing
the hypothesis $s=0$ within this model. Risk-optimizing penalties
$pen$ proposed in \cite{MR1848946} satisfy the conditions of
Definition \ref{GeneralTheoryDef2} (after the change of notations
$pen(m) = \pi(m,n);$ for the explicit expressions of $pen'$s see the
original paper). Therefore, $\| \widehat{s}_{\widehat{m}} \|^2$ is,
in our terminology, the data-driven NT-statistic. As follows from
the consistency Theorem \ref{GeneralTheoryTh5}, $\|\widehat{s}_{\widehat{m}} \|^2$ can be used for testing $s=0$ and
has a good range of consistency, even though this particular penalty probably does not lead to adaptively optimal
testing, see \cite{baraud2003adaptive}. $\Box$

Many other examples of applications of the consistency theorem can be found in \cite{dissertation} and \cite{langovoy_report_2009-007}.

\begin{remark}\label{GNT-Remark17}
In general, it seems to be possible to use the idea of a score
process and some other technics from \cite{MR2281882} in order to
construct and analyze NT-statistics. This can be seen by
the fact that such applications as in Example 6 naturally
appear in both papers. The difference with the above paper would be
that we prefer to use test statistics of the form
(\ref{GeneralTheory4}) rather than integrals or suprema of score
processes.

In semi- and nonparametric models, generalized likelihood ratios
from \cite{1029.62042} and \cite{liliang}, as well as different
modifications of empirical likelihood, could also be a powerful tool
for constructing NT-statistics.
\end{remark}

\noindent {\bf Acknowledgments.} Author would like to thank Fadoua
Balabdaoui, Shota Gugushvili and Axel Munk for helpful discussions. Most of this research was done at Georg-August-University of G\"{o}ttingen, Germany.  \\

\bibliographystyle{unsrtnat}
\bibliography{NT_Bibliography}

\smallskip

\noindent {\bf Appendix.}

\begin{proof}(Proposition \ref{GeneralTheoryTh2}). By the law of large
numbers, as $n \rightarrow \infty\,,$

\begin{equation}\label{GeneralTheory7}
\frac{1}{n}\,\sum_{i=1}^{n} l_K (Y_i)\,\rightarrow_P \,C_P \neq 0.
\end{equation}

\noindent We get

\begin{eqnarray}\label{GeneralTheory8}
T_K & = & \biggr\{ \frac{1}{\sqrt{n}}\, \sum_{i=1}^{n}
\overrightarrow{l} (Y_i) \biggr\} \;L_k\; {\biggr\{
\frac{1}{\sqrt{n}}\, \sum_{i=1}^{n} \overrightarrow{l} (Y_i)
\biggr\} }^{T} \nonumber\\
    & \geq & \lambda^{(k)}_K {\biggr\|
\frac{1}{\sqrt{n}}\, \sum_{i=1}^{n} \overrightarrow{l} (Y_i)
\biggr\|}^2 \nonumber\\
    & \geq & \lambda^{(k)}_K \cdot
\frac{1}{n}\,{\biggr(\sum_{i=1}^{n} l_K (Y_i) \biggr)}^2 \,.
\end{eqnarray}

\noindent By (\ref{GeneralTheory7})

\begin{eqnarray}
T_K - \pi(K,n) & \geq & n\lambda^{(k)}_K \cdot {\Bigr(
\frac{1}{n}\,\sum_{i=1}^{n} l_K (Y_i) \Bigr)}^{\!\!\!2} - \pi(K,n)
\nonumber\\
               &  =   & n\lambda^{(k)}_K \bigr(C_K^2 + o_P (1) C_K \bigr)\,-
               \pi(K,n)\nonumber\\
               &  =   & \,n \lambda^{(k)}_K C_K^2 +\,o_P \bigr(n\lambda^{(k)}_K
\bigr)\,- \pi(K,n)\,,\nonumber
\end{eqnarray}

\noindent and, because $K$ and $C_K$ are constants determined by
fixed $P,$ condition (\ref{GeneralTheory6}) yields

\begin{equation}\label{GeneralTheory10}
T_K - \pi(K,n)\,\rightarrow_P \, \infty \quad\mbox{as}\quad n
\rightarrow\infty \,.
\end{equation}

\noindent On the other hand, by (\ref{GeneralTheory34})

$$\biggr( \frac{1}{\sqrt{n}}\,\sum_{i=1}^{n} l_1 (Y_i), \ldots ,
\frac{1}{\sqrt{n}}\,\sum_{i=1}^{n} l_{K-1}
(Y_i),\biggr)\,\rightarrow_P \, \mathcal{N}\,,$$ where $\mathcal{N}$
is a $(K-1)-$dimensional multivariate normal distribution with the
expectation vector equal to zero. This implies that $T_k=O_P (1)$
for all $k=1, 2, \ldots , K-1,$ because

$$T_k \leq \lambda^{(k)}_1 \, {\biggr\| \frac{1}{n}\,\sum_{i=1}^{n} l
(Y_i) \biggr\|}^2 \,=\,\lambda^{(k)}_1 O_P (1)\,=\,O_P (1)$$ and
$\lambda^{(1)}_1,$ $\lambda^{(2)}_1, \ldots ,$ $\lambda^{(K-1)}_1$
are constants and $K < \infty.$ Now by (\ref{GeneralTheory10})

$$\lim_{n \rightarrow \infty} \,\sum_{k=1}^{K-1} P\,\bigr(T_k - \pi (k,n) \geq T_K - \pi
(K,n)\bigr)\,=\,0\,.$$

\noindent But for $d(n)\geq K$

$$P(S < K) \leq \,\sum_{k=1}^{K-1} P\,\bigr(T_k - \pi (k,n) \geq T_K - \pi
(K,n)\bigr)\,,$$ and the theorem follows.
\end{proof}

Because of assumption $\langle A \rangle$ we can prove the
following lemma.

\begin{lemma}\label{GeneralTheoryLemma1}
$$P \biggr( \biggr| \frac{1}{n}\,\sum_{i=1}^{n} l_K
(Y_i) \biggr| \, \leq \sqrt{\frac{x}{\lambda_K n\,}} \,\biggr)
\,=\,O \biggr(\frac{1}{\,r_n}\biggr)\,.$$
\end{lemma}

\begin{proof}
Denote $x_n := \sqrt{\frac{x}{\lambda_K n\,}}$, and remember that by
$\langle C \rangle$ we have $E_P\,l_K (Y_i) = C_K.$ Obviously, $x_n
\rightarrow 0$ as $n \rightarrow \infty .$ We have

$$P \biggr( \biggr| \frac{1}{n}\,\sum_{i=1}^{n} l_K
(Y_i) \biggr| \, \leq x_n \,\biggr) \,=\, P \biggr( -x_n \,\leq\,
\frac{1}{n}\,\sum_{i=1}^{n} l_K (Y_i)  \, \leq x_n \,\biggr)$$

$$=\,P \biggr( -x_n - C_K \,\leq\,
\frac{1}{n}\,\sum_{i=1}^{n} \bigr(l_K (Y_i) - E_P l_K (Y_i)\bigr)
\, \leq x_n - C_K \,\biggr) \,.$$

\noindent Here we get two cases. First, suppose $C_K > 0.$ Then we
continue as follows:

$$P \biggr( -x_n - C_K \,\leq\,
\frac{1}{n}\,\sum_{i=1}^{n} \bigr(l_K (Y_i) - E_P l_K (Y_i)\bigr)
\, \leq x_n - C_K \,\biggr) \quad\quad\quad\quad\quad\quad$$

\begin{eqnarray}
 \quad\quad\quad & \leq & \,P \biggr( \frac{1}{n}\,\sum_{i=1}^{n} \bigr(l_K (Y_i) - E_P
l_K (Y_i)\bigr) \, \leq x_n - C_K \,\biggr)\nonumber\\
 \quad\quad\quad & \leq & \,P \biggr( \biggr|\frac{1}{n}\,\sum_{i=1}^{n} \bigr(l_K (Y_i)
- E_P l_K (Y_i)\bigr)\biggr| \, \geq \bigr| x_n - C_K
\bigr|\biggr)\nonumber
\end{eqnarray}

\noindent (for all $n \geq$ some $n_K$)

\begin{eqnarray}
 \quad\quad\quad & \leq & \,P \biggr( \biggr|\frac{1}{n}\,\sum_{i=1}^{n} \bigr(l_K (Y_i) -
E_P l_K (Y_i)\bigr)\biggr| \, \geq \, \frac{C_K}{2} \biggr) \,
=\,O \biggr(\frac{1}{\,r_n}\biggr)\nonumber
\end{eqnarray}

\noindent by $\langle A \rangle,$ and so we proved the lemma for
the case $C_K > 0.$ In case if $C_K < 0,$ we write

$$\,P \biggr( -x_n - C_K \,\leq\,
\frac{1}{n}\,\sum_{i=1}^{n} \bigr(l_K (Y_i) - E_P l_K (Y_i)\bigr)
\, \leq x_n - C_K \,\biggr) \quad\quad\quad\quad\quad\quad$$

\begin{eqnarray}
 \quad\quad\quad & \leq & \,P \biggr( \frac{1}{n}\,\sum_{i=1}^{n} \bigr(l_K (Y_i) - E_P l_K
(Y_i)\bigr) \, \geq - x_n - C_K \,\biggr)\nonumber
\end{eqnarray}

\noindent and then we proceed analogously to the previous case.
\end{proof}

\begin{proof} (Theorem \ref{GeneralTheoryTh3}). Let $x>0.$
Since $T_j > T_K$ if $j > K$ and (\ref{GeneralTheory5}) holds, we
get by Proposition \ref{GeneralTheoryTh2} that

\begin{eqnarray}
P(T_S \leq x)\, & = & \,\sum_{j=K}^{d(n)} P(T_j \leq x, \,S=j) +
o(1) \nonumber\\
                & \leq & d(n) \, P(T_K \leq x) + o(1) \nonumber\\
                & \leq & d(n) \, P \biggr( \lambda_K \, \frac{1}{n}\,
{\biggr(\sum_{i=1}^{n} l_K (Y_i) \biggr)}^2 \, \leq x \biggr) +
o(1)\nonumber\\
                &  =   & d(n) \, P \biggr( \biggr| \frac{1}{n}\,\sum_{i=1}^{n} l_K
(Y_i) \biggr| \, \leq \sqrt{\frac{x}{\lambda_K n\,}} \,\biggr) +
o(1) \,.\nonumber
\end{eqnarray}

\noindent Now by Lemma \ref{GeneralTheoryLemma1} and
(\ref{GeneralTheory11}) we get

$$P(T_S \leq x) \,=\,O \biggr(\frac{d(n)}{\,r_n}\biggr)\,+o(1)\,=\,o(1)\,.$$
\end{proof}

\begin{proof}(Theorem \ref{GeneralTheoryTh4}).
If $S \geq K,$ then $T_k - T_1 \geq \pi (k,n) - \pi (1,n)$ for
some $K \leq k \leq d(n)$ and so, equivalently,

\[
\biggr\{ \frac{1}{\sqrt{n}}\, \sum_{i=1}^{n} l (Y_i) \biggr\}
\;L\; {\biggr\{ \frac{1}{\sqrt{n}}\, \sum_{i=1}^{n} l (Y_i)
\biggr\}
}^{T}\quad\quad\quad\quad\quad\quad\quad\quad\quad\quad\quad\quad\quad\quad\quad
\]
\vskip -0.4cm

\begin{equation}\label{GeneralTheory17}
- \,{\biggr\{ \frac{1}{\sqrt{n}}\, \sum_{i=1}^{n} l_1 (Y_i)
\biggr\} }^{2} {\{E_0 {[l_1 (Y)]}^{T} l_1 (Y) \}}^{-1} \geq \,\pi
(k,n) - \pi (1,n)
\end{equation}

\noindent for some $K \leq k \leq d(n),$ where $l = (l_1, l_2,
\ldots , l_k).$ We can rewrite (\ref{GeneralTheory17}) in terms of
the notation (\ref{GeneralTheory12})-(\ref{GeneralTheory15}) as
follows:

\begin{equation}\label{GeneralTheory18}
(\sqrt{\,n}\,\,\,\overline{l}_1, \ldots ,
\sqrt{\,n}\,\,\,\overline{l}_k)\, L\,
{(\sqrt{\,n}\,\,\,\overline{l}_1, \ldots ,
\sqrt{\,n}\,\,\,\overline{l}_k)}^{T}\quad\quad\quad\quad\quad\quad\quad\quad\quad\quad
\end{equation}

\vskip -0.6cm

$$=\,n\, (\overline{l}_1, \ldots ,
\overline{l}_k)\, L\, {(\overline{l}_1,
 \ldots , \overline{l}_k)}^{T} \,\geq
\frac{n\,{\overline{l}_1}^2}{E_0 \,{l_1}^2} + \bigr(\pi (k,n) -
\pi (1,n) \bigr),$$ for some $K \leq k \leq d(n).$ Denote $\Delta
(k,n) := \pi (k,n) - \pi (1,n);$ then with the help of
(\ref{GeneralTheory15}) we rewrite (\ref{GeneralTheory18}) as

\begin{equation}\label{GeneralTheory19}
n\,Q_k (\overline{l}) \, \geq \, \Delta (k,n) +
\frac{n\,{\overline{l}_1}^2}{E_0 \,{l_1}^2}\,,
\end{equation}

\noindent for some $K \leq k \leq d(n).$ Clearly,

\begin{eqnarray}
 P_0(S \geq K) & \leq & P_0 \bigr((\ref{GeneralTheory17}) \,\,\mbox{holds for some}\,\,K \leq k \leq
 d(n)\bigr)\nonumber\\
               &  =   & P_0 \bigr((\ref{GeneralTheory19}) \,\,\mbox{holds for some}\,\,K \leq k \leq
 d(n)\bigr)\nonumber\\
               & \leq & P_0 \bigr( n\,Q_k (\overline{l}) \, \geq \, \Delta (k,n)\,\,\mbox{for some}\,\,
 K \leq k \leq d(n)\bigr)\,.\nonumber
\end{eqnarray}

\noindent But now by condition $\langle B2 \rangle$ we have

\begin{eqnarray}
 P_0(S \geq K) & \leq & P_0\bigr( n\,Q_k (\overline{l}) \, \geq \,
 \Delta (k,n)\,\,\mbox{for some}\,\, K \leq k \leq
 d(n)\bigr)\nonumber\\
               & \leq & \sum_{k=K}^{d(n)} P_0\biggr( n\, Q_k (\overline{l}) \, \geq \,
 \Delta (k,n) \biggr)\nonumber\\
               & \leq & \sum_{k=K}^{d(n)} \varphi \bigr(k; \Delta (k,n)\bigr)\,,
\end{eqnarray}

\noindent if only $d(n) \,\leq\, \min \{u_n, m_n\}$ (see Definition
\ref{GeneralTheoryDef4}). Thus, because of the condition $\langle B2 \rangle$, for
each $\varepsilon > 0$ there exists $K = K_{\varepsilon}$ such that
for all $n > n(\varepsilon)$ we have $P_0(S \geq K) \leq \varepsilon
,$ i.e. $S=O_{P_0}(1).$

Now, by standard inequalities, it is possible to show that
$T_S=O_{P_0}(1).$ Let us write for an arbitrary real $t > 0$

\begin{eqnarray}\label{SimpleDeconvolution12}
 P_0 (|T_S| \geq t) & = & \sum_{m=1}^{K_{\varepsilon}} P_0 (|T_m| \geq t; \,
             S=m)\nonumber\\
                            &   & \quad + \sum_{m=K_{\varepsilon}+1}^{d(n)} P_0 (|T_m| \geq t; \,
             S=m)\nonumber\\
                            & \leq & \sum_{m=1}^{K_{\varepsilon}} P_0 (|T_m| \geq
                            t) + \sum_{m=K_{\varepsilon}+1}^{d(n)} P_0
                            (S=m)\nonumber\\
                            & = & \sum_{m=1}^{K_{\varepsilon}} P_0 (|T_m| \geq
                            t) + P_0 (S \geq
                            K_{\varepsilon}+1)\nonumber\\
                            & \leq & \sum_{m=1}^{K_{\varepsilon}} P_0 (|T_m| \geq
                            t) + \varepsilon\nonumber\\
                            & =: & R(t) + \varepsilon .\nonumber
\end{eqnarray}

\noindent For $t \rightarrow \infty$ we have $P_0 (|T_m| \geq t)
\rightarrow 0$ for every fixed $m,$ so $R(t) \rightarrow 0$ as $t
\rightarrow \infty .$ Now it follows that for arbitrary
$\varepsilon > 0$

\[
\overline{\lim_{t \rightarrow \infty}} P_0 (|T_S| \geq t) \leq
\varepsilon ,
\]
\noindent therefore
\[
\overline{\lim_{t \rightarrow \infty}} P_0 (|T_S| \geq t) = 0
\]
\noindent and
\[
\lim_{t \rightarrow \infty} P_0 (|T_S| \geq t) = 0 .
\]
This completes the proof.
\end{proof}

\begin{proof}(Theorem \ref{GeneralTheoryTh5}). Follows from Theorems \ref{GeneralTheoryTh3} and
\ref{GeneralTheoryTh4} and our definition of consistency.
\end{proof}

In the next proof we will need the following theorem from
\cite{MR0312546}.

\begin{theorem}\label{Prohorov}
Let $Z_1,$ $\ldots ,$ $Z_n$ be i.i.d. random vectors with values
in $\mathbb{R}^k.$ Let $E Z_i =0$ and let the covariance matrix of
$Z_i$ be equal to the identity matrix. Assume $\|Z_1\|_k \leq L$
a.e. Then, for $2k \leq y^2 \leq n L^{-2},$ we have

\[
Pr \biggr( \| n^{-1/2} \sum_{i=1}^{n} Z_i \|_k \geq y \biggr) \leq
\frac{150210}{\Gamma (k/2)}\,
{\biggr(\frac{y^2}{2}\biggr)}^{\frac{k-1}{2}} \exp \biggr\{ -
\frac{y^2}{2} \biggr(1- \eta_n \biggr) \biggr\}\,,
\]

\noindent where $0 \leq \eta_n \leq L y n^{-1/2}.$
\end{theorem}

\begin{proof}(Proposition \ref{GeneralTheoryTh9})
Here $T_S$ is an NT-statistic with $L_k= E_{k \times
k}$ and $\lambda^{(k)}_1 = \ldots = \lambda^{(k)}_k =1.$ Therefore
Theorem \ref{GeneralTheoryTh4} is applicable. Put $s(k,n)=\sqrt{\,2k\,},$ $t(k,n)= 0.5 \sqrt{\,n\,}\,
M(k)^{-1}.$ The Prohorov inequality is applicable if $M(k)\, \pi(k,n) \leq \sqrt{\,n}$ and $M^2(k)\, \pi (k,n) \leq n$ for all $k
\leq d(n);$ therefore assumption (\ref{GeneralTheory39})
guarantees that the Prohorov inequality is applicable and,
moreover, that $\langle B2 \rangle$ holds with

\begin{equation}\label{GeneralTheory22}
\varphi (k; y)\,=\, \frac{150210}{\Gamma (k/2)}\,
{\biggr(\frac{y^2}{2}\biggr)}^{\frac{k-1}{2}} \exp \biggr\{ -
\frac{y^2}{4} \biggr\}\,.
\end{equation}

\noindent Since $\varphi$ is exponentially decreasing in $y$ under
(\ref{GeneralTheory39}), it is a matter of simple calculations to
prove that $\langle B1 \rangle$ is satisfied with $u_n = d(n)$ for any sequence
$\{ d(n) \}$ such that (\ref{GeneralTheory39}) holds.

\end{proof}

\end{document}